\documentclass[12pt]{article}
\usepackage{amsmath,amssymb,amsthm, blindtext, mathtools}
\usepackage[mathscr]{eucal}
\usepackage[usenames,dvipsnames]{xcolor}
\usepackage{enumerate}

\DeclareMathOperator{\dist}{dist}
\newcommand{\bndry}{b}

\newcommand{\CintD}{\mathbf C_{\domain}}

\newcommand{\domain}{D}
\newcommand{\realdomain}{\domain}
\newcommand{\dee}{\partial}
\newcommand{\deebar}{\overline\dee}

\newcommand{\neu}{\mathfrak{n}}
\newcommand{\reu}{\mathfrak{r}}
\newcommand{\Hs}{\mathcal H}
\newcommand{\Hsa}{\mathcal H_\alpha}

\newcommand{\n}{\nu}

\newcommand{\Bzj}{\tilde{b}_{z}}

\long\def\symbolfootnote[#1]#2{\begingroup
\def\thefootnote{\fnsymbol{footnote}}\footnote[#1]{#2}\endgroup}

\newtheorem{thm}{Theorem}[section]
\newtheorem{prop}[thm]{Proposition}
\newtheorem{lem}[thm]{Lemma}
\newtheorem{cor}[thm]{Corollary}

\theoremstyle{definition}

\newtheorem{defn}[thm]{Definition}

\newtheorem{example}[thm]{Example}

\theoremstyle{remark}

\newtheorem{rem}[thm]{Remark}
\usepackage{comment}

\title{Boundary value problems
for holomorphic functions\\
on Lipschitz planar domains}
\author{William Gryc, Loredana Lanzani\footnote{supported in part by the National Science Foundation, award no. DMS-1901978, and a Simons Foundation Travel Support for Mathematicians, award no. 919763.}, Jue Xiong, Yuan Zhang}
\renewcommand{\thefootnote}{\fnsymbol{footnote}} 
\date{}
\begin{document}
\maketitle
\begin{abstract}
 We study the $\bar\partial $ equation subject to various boundary value conditions on bounded simply connected Lipschitz domains  $\domain\subset\mathbb C$: 
 for the Dirichlet problem with datum in $L^p(\bndry\domain, \sigma)$, this is simply a restatement of the fact that members of the holomorphic Hardy spaces are uniquely and completely determined by their boundary values. Here we identify the maximal data spaces and obtain estimates in the maximal $p$-range for the Dirichlet, Regularity-for-Dirichlet, Neumann, and Robin boundary conditions for 
 $\bar\partial$.

\end{abstract}

\section{Introduction}

Let $\domain$ be a
bounded,
rectifiable,
simply connected 
domain in $\mathbb C$ whose boundary is endowed with the induced Lebesgue measure $\sigma$.
We denote by
$\mathcal E^p(\domain)$ the Smirnov class
$$\displaystyle{ \mathcal E^p(\domain): =\left\{F\in\vartheta(\domain): \|F\|_{\mathcal E^p(\domain)}^p :=\sup_{j\in\mathbb{N}}\,\int\limits_{\bndry\domain_j} |F(\zeta)|^p d\sigma_j(\zeta) <\infty \right\},\quad 0<p< \infty },$$
(with the standard modification for $p=\infty$) where $\vartheta(\domain)$ is the 
set of holomorphic functions on $\domain$, and $\{\domain_j\}_{j\in\mathbb{N}}$ is (any) exhaustion of $\domain$ by rectifiable subdomains.  Each $F\in \mathcal E^p(\domain)$ has a non-tangential boundary value $\dot F (\zeta)$ at $\sigma$-a.e. $\zeta\in\bndry\domain$,
 see \eqref{D:Fdot}.  As is well known, such boundary values 
determine a proper subspace 
of the Lebesgue space $L^p(\bndry\domain, \sigma)$, see \cite[Theorem 10.3]{Duren}.  The following congruence of Banach spaces (that is, set identity and norm equivalence) was proved by D. Jerison and C. Kenig \cite{JK} whenever $D$ is a chord-arc domain:
\begin{equation}\label{E:JK}
\displaystyle{ \mathcal E^p(\domain) = \Hs^p(\domain),\quad 0<p\leq\infty.}
\end{equation}
Here
  $\Hs^p(\domain)$ is the holomorphic Hardy space:
\begin{equation}\label{D:Hp}
\Hs^p(\domain):= \{F\in\vartheta (D):\, F^*\in L^p(\bndry D, \sigma)\}\ \ \text{with norm}\ \ 
\|F\|_{\Hs^p(\domain)}:= \|\dot F\|_{L^p(\bndry\domain, \sigma)}
\end{equation}
where $F^*$ denotes the non-tangential maximal function of $F$, see \eqref{D:Fstar}.
In view of the congruence \eqref{E:JK}, we henceforth adopt the notation $\Hs^p(\domain)$ and $h^p(\bndry\domain)$, respectively,  for the interior and boundary values of such spaces; hence
\begin{equation}\label{D:hp}
h^p(\bndry\domain) := \left\{\dot F\ :\ F\in \Hs^p(\domain)\right\}\, \subsetneq\, L^p(\bndry\domain, \sigma).
\end{equation}
One way to understand the relationship between $\Hs^p(\domain)$ and $h^p(\bndry\domain)$ 
  is to reframe both spaces within the context of a {\bf Dirichlet boundary value problem for
  $\bar\partial$}
  with datum in
    $L^p(\bndry\domain, \sigma)$, namely
\begin{equation}\label{E:D-Dbar}
     \left\{
      \begin{array}{lcll}
      \bar\partial F(z) &= &0 & 
            \ z\in \domain;\\ \\
      \dot F(\zeta) &= &f(\zeta) & 
          \ \sigma\text{-a.e.}\ \zeta\in \bndry\domain;\\ \\
      F^*&\in& L^p(\bndry\domain, \sigma).
      \end{array}
\right.
\qquad \text{where we are given}\quad f\in L^p(\bndry\domain, \sigma)
\end{equation}
 This problem is uniquely solvable if and only if $f$ belongs to $h^p(\bndry\domain)$:
 in such case, the solution $F$ must lie in $\Hs^p(\domain)$.
  Put another way, 
we have that $h^p(\bndry\domain)$
  is the maximal data set for the problem \eqref{E:D-Dbar}, and
$\Hs^p(\domain)$
is the solution space. See also \cite{Begher} for a related result in the case when $\domain=\mathbb{D}$ (the unit disk).

 If $\domain$  is a  Lipschitz domain (that is, the
 topological boundary $\bndry\domain$ is locally the graph of a Lipschitz  function, see Definition \ref{de}), the above fact is stated more precisely as
 \begin{thm}\label{T:D-Dbar}
 Let $\domain\subset \mathbb C$ be a bounded simply connected Lipschitz domain  and let $p>0$.
 Given  $f\in L^p(\bndry\domain, \sigma)$, we have that
 \eqref{E:D-Dbar}
   is solvable if and only if $f\in h^p(\bndry\domain)$, in which case \eqref{E:D-Dbar} has a unique solution and it lies in $\Hs^p(\domain)$.
  Furthermore,
\begin{itemize}
\item if  $p\geq 1$ and $f\in h^p(\bndry\domain)$, then the solution $F$ admits the representation 
$$F(z) = \mathbf C_\domain f (z),\quad z\in \domain$$
where $\mathbf C_\domain f$ is the Cauchy integral for $\domain$ acting on $f$;
\item if $1<p<\infty$ and $f\in h^p(\bndry\domain)$, we also have that \begin{equation}\label{E:Dpbound}
\|F^*\|_{L^p(\bndry\domain, \sigma)}\lesssim
 \|f\|_{L^p(\bndry\domain, \sigma)}
 \end{equation}
and the implied constant depends only on $\domain$ and $p$.
\end{itemize}
\end{thm}

\vskip0.1in
Readers familiar with the analogous result
for the Laplace operator
might be surprised to see that conclusion \eqref{E:Dpbound} in Theorem \ref{T:D-Dbar} holds for $p$ in the full range: $1<p<\infty$ rather than
 $2-\delta <p<\infty$, see \eqref{E:DL} and  Theorem \ref{T:DL} below. This reflects the fact that
  the effective
   data set for \eqref{E:D-Dbar}, namely
$h^p(\bndry\domain)$, is a very small portion
of $L^p(\bndry D, \sigma)$
(the data set for \eqref{E:DL}).
It follows from Theorem \ref{T:D-Dbar} that if the 
 boundary data of the Dirichlet problem for Laplace operator \eqref{E:DL}
 is taken in
$h^p(\bndry D, \sigma)$,  
then the existence of a unique solution (with estimates) is guaranteed for any $1<p<\infty$. Analogous statements are true for all the
 boundary conditions considered below.
\vskip0.1in

 Here we
seek to determine the
maximal data sets for
 the \textbf{Regularity; Neumann, and Robin problems} for $\deebar$ with data in
 Lebesgue or Sobolev spaces. Namely,
\vskip0.1in
\begin{equation}
\label{E:D-Dbar-reg}
 \left\{
      \begin{array}{lcll}
      \bar\partial F(z) &= &0 & 
            \ z\in \domain;\\ \\
      \dot F(\zeta) &= &f(\zeta) & 
           \ \sigma\text{-a.e.}\ \zeta\in \bndry\domain;\\ \\
      (F')^*&\in& L^p(\bndry\domain, \sigma)
      \end{array}
\right.
\end{equation}
given $f \in W^{1, p}(\bndry\domain, \sigma)$. Also

\begin{equation}\label{NdbarIntro}
     \left\{
      \begin{array}{lcll}
      \bar\partial G &= &0 & \text{in} \ \ \domain;\\ \\
       \displaystyle{\frac{\partial G}{\partial n}(\zeta)} &= &g(\zeta) & \text{for}\ \sigma\text{-a.e.}\ \zeta\in \bndry\domain\,
       \\ \\
      (G')^*&\in& L^p(\bndry\domain, \sigma)
      \end{array}
\right.
\end{equation}
given $g\in L^p(\bndry\domain, \sigma)$ subject to the compatibility condition: 
$\displaystyle{\int\limits_{\bndry\domain}g\, d\sigma =0}$, and
\vskip0.2in
\begin{equation}\label{RdbarIntro}
     \left\{
      \begin{array}{lcll}
      \bar\partial G &= &0 & \text{in} \ \ \domain;\\ \\
       \displaystyle{\frac{\partial G}{\partial n}(\zeta) +b(\zeta) \dot G(\zeta)}&= & r(\zeta)
        & \text{for}\ \sigma\text{-a.e.}\ \zeta\in \bndry\domain
        \\ \\ (G')^*&\in& L^p(\bndry\domain, \sigma)\, 
      \end{array}
\right.
\end{equation}
\vskip0.12in

\noindent given $b$ and $r$ in $L^p(\bndry\domain, \sigma)$, and with the
Robin coefficient $b$ subject
to compatibility conditions that will be specified below. 
  In both problems,  $G'$ denotes the complex derivative of $G$:
$$
G'(z) = \frac{\dee G}{\dee z}(z),\quad z\in D
$$
whereas
\begin{equation}\label{E:normal-deriv-hol}
\frac{\dee G}{\dee n}(\zeta):=\ -i T(\zeta) \dot{(G')}(\zeta) \ =\ \lim_{\stackrel{z\to \zeta}{ z\in \Gamma (\zeta)}}\langle \nabla G (z), n(\zeta)\rangle_{\mathbb R}\quad \sigma-\text{a.e.}\ \zeta \in\bndry\domain. 
\end{equation}
Here $\Gamma(\zeta)$ is the coordinate (or regular) cone at $\zeta\in b\domain$, see Definition \ref{D:ap}, $T(\zeta):= t_1(\zeta) + i t_2(\zeta)$ is the positively oriented, complex tangent vector at $\zeta\in\bndry\domain$, and $n(\zeta) := -iT(\zeta)$ is the complex outer unit normal vector which, depending on context,
we may interpret as the
vector $(t_2(\zeta), -t_1(\zeta))\in\mathbb R^2$.
\vskip0.1in

Our candidate solution space for all of these boundary conditions
  is the {\bf holomorphic Sobolev-Hardy space}
     \begin{equation}\label{D:H1q}
    \Hs^{1, p}(\domain): = \{G\in \mathcal \vartheta(\domain): G'\in \Hs^p(\domain)\},\quad p>0\, .
\end{equation}
It turns out that $ \Hs^{1, p}(\domain)$ is a Banach space for any $p\geq 1$ 
 in which case
 we have that
\begin{equation}\label{E:inclusion}
 \Hs^{1, p}(\domain)\, \subseteq \, \Hs^p(\domain)
\end{equation}
moreover for $1<p<\infty$ the set inclusion \eqref{E:inclusion} is also an embedding of Banach spaces, see Corollary \ref{HSsubsetHardyC} and Lemma \ref{4i}.
Hence for any $p\geq 1$ the members of  $\Hs^{1, p}(\domain)$ possess non-tangential boundary values which are
 in $L^p(\bndry\domain, \sigma)$. We write
\begin{equation}\label{D:h1q}
h^{1, p}(\bndry\domain): = \left\{\dot G: \ G\in \mathcal H^{1, p}(\domain)\right\}\,,\quad 1\leq p\leq\infty
\end{equation}
 and we show that $h^{1,p}(\bndry\domain) \subset W^{1, p}(\bndry\domain, \sigma)$, the Sobolev space for $\bndry\domain$, see Proposition \ref{d}. 
 Thus any
  $h\in h^{1, p}(\bndry\domain)$ possesses a tangential derivative $\dee_T h$ $\sigma$-a.e. $\bndry\domain$.

  \vskip0.1in
  Our first main result states that  $h^{1, p}(\bndry\domain)$ is the optimal data space for the Regularity problem for $\deebar$. That is,
  \vskip0.1in
  
  \begin{thm}\label{T:D-Dbar-reg}
 Let $\domain$ be a bounded simply connected Lipschitz domain  and let $p\geq 1$.
  Given  $f\in W^{1, p}(\bndry\domain, \sigma)$, we have that   \eqref{E:D-Dbar-reg}
  is solvable if and only if $f\in h^{1,p}(\bndry\domain)$, in which case \eqref{E:D-Dbar-reg} has a unique solution $F$ and it lies in $\Hs^{1,p}(\domain)$.
If $1<p<\infty$ and $f\in h^{1,p}(\bndry\domain)$ we also have that \begin{equation*}\label{E:Dpbound-reg}
\|F^*\|_{L^\infty(\bndry\domain, \sigma)} +
\|(F')^*\|_{L^p(\bndry\domain, \sigma)} 
\lesssim
 \|f\|_{W^{1,p}(\bndry\domain, \sigma)}
 \end{equation*}
and the implied constant depends only on $\domain$ and $p$.
\end{thm}

 We next consider the following two subspaces of $L^p(\bndry\domain, \sigma)$: 
\begin{equation}\label{D:Ndata}
\neu^{p}(\bndry\domain):= \left\{ g:= -i\dee_T h: \ h\in h^{1, p}(\bndry\domain)\right\},\quad 1\leq p\leq\infty
\end{equation}
and
 \begin{equation}\label{D:Rdata}
\reu^p_b(\bndry\domain): =\left \{ r:= -i\dee_T h + bh
: h\in h^{1,p}(\bndry\domain) \right \}\, \quad 1\leq p\leq \infty.
\end{equation}
That is 
\begin{equation*}\label{D:Ndata-a}
\neu^{p}(\bndry\domain):= \left\{ g:= -i\dee_T \dot{G}: \ G\in  \mathcal H^{1, p}(\domain)\right\},\quad 1\leq p\leq\infty
\end{equation*}
and
 \begin{equation*}\label{D:Rdata-a}
\reu^p_b(\bndry\domain): =\left \{ r:= -i\dee_T \dot{H} + b\dot{H}
: H\in \Hs^{1,p}(\bndry\domain) \right \}\, \quad 1\leq p\leq \infty.
\end{equation*}

\vskip0.1in

Our second main result gives  that $\neu^p(\bndry\domain)  $ is the maximal data space for the Neumann  problem for the $\deebar$ operator. Specifically,

\begin{thm}\label{T:NDbar}
Let $\domain\subset\mathbb C$ be a bounded simply connected Lipschitz domain and let 
$p\geq 1$.
Given $g\in L^p(\bndry\domain, \sigma)$ we have that
 \eqref{NdbarIntro}
 is solvable  if and only if  $g\in\neu^{p}(\bndry\domain)$, in which case 
all  solutions are in
 $\Hs^{1,p}(\domain)$ and
differ by an additive  constant.
 Furthermore,
\begin{itemize}
\item 
the complex derivative of each solution $G$ admits the representation
$$G'(z) = \CintD (i\overline{T}g)(z), \quad z\in\domain.$$
 Here 
 $\CintD (i\overline{T}g)$ is the Cauchy integral for $\domain$ acting on the pointwise product 
 $$i\,\overline{T}(\zeta) g(\zeta), \quad \sigma\text{-a.e.}\ \zeta\in\bndry\domain$$
 where $T(\zeta)$ is the unit tangent vector at $\zeta$;
\item 
If $1<p<\infty$ and $g\in \neu^p(\bndry\domain)$ we have
\begin{equation*}\label{E:Nqbound}
\|(G')^*\|_{L^p(\bndry\domain, \sigma)}\lesssim
 \|g\|_{L^p(\bndry\domain, \sigma)}
 \end{equation*}
 for each solution $G$. The implied constant depends only on $\domain$ and $p$;
\item 

If $1<p<\infty$ and $g\in \neu^p(\bndry\domain)$, then for any $\alpha\in D$ there is a unique solution $G_\alpha$ in
 \begin{equation*}\label{E:H1q-alpha}
 \Hsa^{1,p}(\domain):= \{G\in \Hs^{1, p}(\domain):\  G(\alpha) =0 \}
 \end{equation*}
  and $G_\alpha\in C^{1-\frac{1}{p} }(\overline{\domain})$ with
 \begin{equation*}\label{E:Nqbound-alpha}
\|G_\alpha^*\|_{C^{1-\frac{1}{p} }(\overline{\domain})}
\lesssim
 \|g\|_{L^p(\bndry\domain, \sigma)}.
 \end{equation*}
 Moreover, $G_\alpha$ admits the representation
 \begin{equation}\label{E:temp11}
 G_\alpha(z)  =\mathbf C_\domain \big(h_0 - h_0(\alpha)\big)(z),\quad z\in \domain
 \end{equation} 
 where
 \begin{equation*}\label{E:def-h_0}
 h_0(\zeta) :=  \int\limits_{\gamma(\zeta_0, \zeta)}\!\!\!\!\!  ig(\eta)d\sigma(\eta)\quad \in\quad  h^{1,p}(\bndry\domain)
 \end{equation*}
 with $\zeta_0\in\bndry\domain$ fixed arbitrarily and 
  \begin{equation}\label{E:path-def}
    \gamma (\zeta_0, \zeta)\ := \left\{
      \begin{array}{lcll}
      \text{the piece of $\bndry\domain$ going from $\zeta_0$ to $\zeta$ in the positive direction} & \text{if}& \zeta_0\neq\zeta\, ,\\
      \emptyset & \text{if} & \zeta_0 =\zeta\, .
      \end{array}
\right.
\end{equation}
Furthermore, for each $\alpha\in\domain$ we have
\begin{equation*}\label{E:temp5}
\|G_\alpha^*\|_{L^\infty(\bndry\domain, \sigma)} \ +\ \|(G_\alpha')^*\|_{L^p(\bndry\domain, \sigma)}\ \lesssim\ 
\|g\|_{L^p(\bndry\domain, \sigma)}, \quad 1<p<\infty.
\end{equation*}
The implied constant depends only on $\domain$, $p$ and $\alpha$.
\end{itemize}

\end{thm}

\vskip0.1in
For the Robin problem for the $\deebar$ operator \eqref{RdbarIntro} we may take the Robin coefficient $b$ to be in $L^p(\bndry\domain, \sigma)$: this is on account of the inclusion $h^{1,p}(\bndry\domain)\subset C^{1 -\frac1p}(\bndry\domain)$ for any $p\geq 1$
 (Proposition \ref{d});
 in particular $h^{1,p}(\bndry\domain) \subset L^\infty(\bndry\domain)$ so that
  $b\,\dot G\in L^p(\bndry\domain, \sigma)$ for any $G\in \Hs^{1,p}(\domain)$ and $b\in L^p(\bndry \domain, \sigma)$.
 In order to guarantee uniqueness of the solution, as well as a representation formula, we further require that
\begin{equation}\label{E:compint}
     \int\limits_{\bndry\domain}\! b(\xi)\, d\sigma(\xi)\ \ne\ 2k\pi \quad\text{for any}\quad  k\in \mathbb Z.
\end{equation}
See Example \ref{rr} for a holomorphic Robin problem which has infinitely many solutions where the Robin coefficient $b$  is not identically 0 and does not satisfy condition \eqref{E:compint}. Note that 
the Neumann condition cannot be regarded as a sub-case of the Robin condition (namely with $b:= 0$) because it does not satisfy \eqref{E:compint}.

 For any $\zeta, \ \xi\ \in\bndry\domain$  we let $\gamma (\zeta, \xi)$ be as in \eqref{E:path-def} (with $\zeta$ and $\xi$ in place of $\zeta_0$ and $\zeta$, respectively).
  Define the operator
  \begin{equation}\label{E:auxOp}
     \displaystyle{ 
       {\mathcal T_b r}(\zeta): = \frac{\displaystyle{i\int\limits_{\bndry\domain}\!\!r(\xi) \,\,e^{\,\,\int\limits_{\gamma(\zeta, \xi)}\!\!\!\!\!\! i\,b(\eta) d\sigma(\eta)}d\sigma(\xi)}}
     {e^{\,\,\int\limits_{\bndry\domain}\!\! i\,b(\xi) d\sigma(\xi)} -1}
    ,\ \ \  \ \zeta\in\bndry\domain.}
  \end{equation}
From \eqref{E:compint}, it is clear the denominator in \eqref{E:auxOp} cannot vanish, and this makes it easy to see that
 
 $$
 \mathcal T_b: \ L^p(\bndry\domain, \sigma)\to L^\infty(\bndry\domain, \sigma),\quad 1\le p\le \infty
 $$
 is bounded with
\begin{equation}\label{sup}
    \|\mathcal T_b r\|_{L^\infty(\bndry\domain, \sigma)}\leq \sigma(bD)^{1-\frac{1}{p}}\, C(\domain, b) \  \|r\|_{L^p(\bndry\domain, \sigma)}
\end{equation}
   where $\sigma(bD)$ is the length of $bD$, and 
 $$
 C(\domain, b) := \, e^{\|b\|_{L^1(\bndry\domain, \sigma)}} \left| e^{\,\,\int\limits_{\bndry\domain}\!\! i\,b(\xi) d\sigma(\xi)} -1\right|^{-1}.
 $$
 We also have that $\mathcal T_b$ is smoothing for $L^p(\bndry\domain, \sigma)$ whenever $1<p<\infty$, see
 Proposition \ref{P:Top} for the precise statement. 
   We may now state our third main result.

\begin{thm}\label{T:RDbar} Let $\domain\subset \mathbb C$ be a bounded simply connected Lipschitz domain and let $p\geq 1$. Suppose that the Robin coefficient $b$ is in $L^p(\bndry\domain, \sigma)$  and satisfies 
condition \eqref{E:compint}.
Given  $r\in L^p(\bndry\domain, \sigma)$ we have that
\eqref{RdbarIntro}
  is      solvable
  if and only if $r\in \reu^p_b(\bndry\domain)$, in which case \eqref{RdbarIntro} has a unique solution and it lies in $\Hs^{1,p}(\domain)$.
   Furthermore,
\begin{itemize}
\item the solution $G$ admits the representation $$G(z) = (\CintD\circ \mathcal T_b)\, r(z),\quad z\in\domain\,,$$
where $ \mathcal T_b$ is defined in \eqref{E:auxOp};
\item for any $1<p<\infty$
 we have that
\begin{equation}\label{E:Rqbound} 
 \|G^*\|_{L^\infty(\bndry\domain, \sigma)}
+\|( G')^*\|_{L^p(\bndry\domain, \sigma)}\lesssim\|r\|_{L^p(\bndry\domain, \sigma)},
 \end{equation}
 where the implied constant depends only on $\domain$, $p$ and the Robin coefficient $b$.
 \end{itemize}
  \end{thm}

\vskip0.1in

\begin{rem}\label{R:simply versus N-connected} Any bounded Lipschitz domain must be finitely connected
(because locally it is a graph domain). A large portion of the harmonic analysis literature for Lipschitz domains requires the ambient domain to be simply connected: sometimes this choice is only a matter of convenience (the results can be extended to any bounded Lipschitz domain by routine arguments). But there are also results whose proofs  rely on conformal mapping (this is the case for e.g., the proof of identity \eqref{E:JK}, see \cite{JK}): extending such results to arbitrary Lipschitz domains would require a detailed analysis of the boundary behavior of the Ahlfors map (the multiply-connected-domain analog of the Riemann map). To avoid these issues here we focus on simply connected 
domains (with a few exceptions in Section \ref{SobolevHardySection}). 
Further results
for multiply connected $\domain$ can be found in \cite[Section 5]{GLZ}. 
\end{rem}

\begin{rem}\label{R:regular cones}We expect that our results may hold in less regular settings than Lipschitz.
Here we restrict the attention to the 
Lipschitz category
because already it
displays several features 
typical of most non-smooth frameworks e.g., the restricted $p$-range for the boundary value problems for the Laplacian, see Theorems \ref{T:DL} and \ref{T:NL} below,
 yet
the proofs are  less technical than those for more irregular settings. 
For instance, for the non-tangential approach region $\Gamma (\zeta)$, $\zeta\in\bndry\domain$, here we may take the so-called {\em regular cone} (or {\em coordinate cone})  given in Definition \ref{D:ap}, which is easier to work with than the corkscrew-like region $\{z\in \domain:\, |z-\zeta |\leq (1+\beta)\dist(z, \bndry \domain)\}$ 
needed to work with
 domains whose boundary  is not a  local-graph.
\end{rem}

We henceforth use the following notation:  two quantities $A$ and $B$ are said to satisfy $A\lesssim B$, if  $A\le CB$ for some constant $C>0$ which may depend only on the domain $\domain$, the exponent $p$ 
and, when relevant, the fixed point $\alpha$, or the Robin coefficient $b$. We say  $A\approx B$ if and only if $A\lesssim B$ and $B\lesssim A$ at the same time.
\vskip0.1in
This paper is organized as follows. In Section \ref{SS:Preliminaries} we collect a few well known features of Lipschitz domains that are relevant here, and we recall the Dirichlet and  Neumann problems for the Laplace operator; we also prove Theorem \ref{T:D-Dbar}. In Section \ref{SobolevHardySection} we present the main properties of the holomorphic Sobolev-Hardy space $\Hs^{1,p}(\domain)$. In Section \ref{S:T:D-Dbar-reg} we prove Theorem \ref{T:D-Dbar-reg} (Regularity problem for $\deebar$). Finally, Theorem \ref{T:NDbar} (Neumann problem for $\deebar$) and Theorem \ref{T:RDbar} (Robin problem for $\deebar$) are proved in Sections \ref{NeumannSection} and  \ref{RobinSection},  respectively. 

\vskip0.1in
\noindent\textbf{Acknowledgement: } This work was launched at the AIM workshop {\em Problems on Holomorphic Function Spaces \& Complex Dynamics},
an activity of the AWM Research Network in Several Complex Variables.
We thank
the American Institute of Mathematics and  the Association for Women in Mathematics for their hospitality and support. We also wish to express our gratitude to Zhongwei Shen for helpful discussions.

%%%%%%%%%%%%%%%%%%%%%%%%%%%%%%%%%%%%%%%%%%%%%%%%%%%%%%%%%%%%%%%%%%%%%%%%%%%%%%%%%%%%%%%%%%%%%%%%%%%%%%%%%%
\section{Preliminaries}\label{SS:Preliminaries}

Throughout this paper the domains under consideration will be Lipschitz domains on $\mathbb C$, as defined below.
\begin{defn}\label{de}
A bounded domain $\domain\subset\mathbb{C}$ with boundary $\bndry\domain$ is called a \emph{Lipschitz domain} if there are finitely many rectangles $\{R_j\}_{j=1}^m$ with sides parallel to the coordinate axes, angles $\{\theta_j\}_{j=1}^m$, and Lipschitz functions $\phi_j:\mathbb{R}\to\mathbb{R}$ such that the collection $\{e^{-i\theta_j}R_j\}_{j=1}^m$ covers $\bndry\domain$ and $(e^{i\theta_j}\domain)\cap R_j=\{x+iy: y > \phi_j(x),\  x\in (a_j, b_j)\}$ for some $a_j<b_j<\infty$. 
We refer to such $R_j$'s as {\em coordinate rectangles}.
\end{defn}

\begin{defn}\label{D:ap}
Let $\domain$ be a Lipschitz domain and fix $\beta>0$. For any $\zeta\in\bndry \domain$, let  
$\{\Gamma(\zeta), \zeta\in \domain\}$ be a family of truncated (one-sided) open 
 cones $\Gamma(\zeta)$ with vertex at $\zeta$ satisfying the following property: for each rectangle $R_j$ 
 in Definition \ref{de}, there exists two cones $\alpha$ and $\beta$, each with vertex at the origin and axis along the $y$ axis such that for $\zeta\in \bndry\domain \cap e^{-i\theta_j}R_j$,
$$ e^{-i\theta_j}\alpha +\zeta\quad \subset\quad  \Gamma(\zeta)\quad \subset\quad  \overline{\Gamma(\zeta)}\setminus \{\zeta\}\quad \subset\quad  e^{-i\theta_j}\beta+\zeta \quad  \subset\quad   \domain\ \cap\  e^{-i\theta_j}R_j. $$
It is well known that for Lipschitz $\domain$, $\Gamma (\zeta)\neq \emptyset$ for any $\zeta \in \bndry\domain$; see e.g.,  \cite{Dah2} or \cite[Section 0.4]{V}. Sometimes  $\Gamma (\zeta)$ is referred to as a {\em regular cone}, or a {\em coordinate cone}, see Remark \ref{R:regular cones}.
\end{defn}

We will need  an approximation scheme of $\domain$ by smooth subdomains constructed by Ne\v{c}as in \cite{Necas}, which we refer to as a \emph{Ne\v{c}as exhaustion of $\domain$}. See also \cite{L1} and \cite[Theorem 1.12]{V}. (Recall that Lipschitz functions are differentiable almost everywhere; thus if $\domain$ is a bounded Lipschitz domain, its boundary $\bndry\domain$
  is a rectifiable Jordan curve
  that admits a (positively oriented) unit tangent vector $T(\zeta)$ for $\sigma$-a.e. $\zeta\in\bndry\domain$.)

\begin{lem}\cite[p. 5]{Necas}\cite[Theorem 1.12]{V}\label{L:NecasExhaustion}
Let $\domain$ be a bounded Lipschitz domain. There exists a family $\{\domain_k\}_{k=1}^\infty$ of smooth domains with $\domain_k$ compactly contained in $\domain$ that satisfy the following:
\begin{enumerate}[(a).]
\item For each $k$ there exists a Lipschitz diffeomorphism $\Lambda_k$ that takes $\domain$ to $\domain_k$ and extends to the boundaries: $\Lambda_k:\bndry\domain\to \bndry\domain_k$ with the property that
\[\sup\{|\Lambda_k(\zeta)-\zeta|:\zeta\in \bndry\domain\}\leq C/k\]
for some fixed constant $C$. Moreover $\Lambda_k(\zeta)\in \Gamma(\zeta)$.
\item There is a covering of $\bndry\domain$ by finitely many coordinate rectangles which also form a family of coordinate rectangles for $\bndry\domain_k$ for each $k$. Furthermore for every such rectangle $R$, if $\phi$ and $\phi_k$
  denote the Lipschitz functions whose graphs describe the boundaries of $\domain$ and $\domain_k$, respectively, in $R$, then $\|(\phi_k)'\|_\infty \leq \|\phi'\|_\infty$ for any $k$; $\phi_k\to\phi$ uniformly as $k\to\infty$, and $(\phi_k)'\to\phi'$ a.e. and in every $L^p((a, b))$ with $(a, b)\subset\mathbb R$ as in Definition \ref{de}.
 \item There exist constants $0<m<M<\infty$ and positive functions (Jacobians) $w_k:\bndry\domain\to [m,M]$ for any $k\in \mathbb N$, such that for any measurable set $F\subseteq \bndry\domain$ and 
 for any measurable function $f_k$ on $\Lambda_k (F)$ the following change-of-variables formula holds:
\[\int\limits_{F}
\!\!
f_k(\Lambda_k(\eta))\,w_k(\zeta)\,d\sigma(\eta) = \int\limits_{ \Lambda_k(F)}
\!\!\!
f_k(\eta_k)\,d\sigma_k(\eta_k).\]
where $d\sigma_k$ denotes arc-length measure on $\bndry\domain_k$.  
Furthermore we have
 $$w_k\to  1 \quad \sigma\text{-a.e.}\ \bndry\domain\ \  \text{and in every}\quad  L^p(\bndry\domain,\sigma)\ \ \text{for any}\ \  1\leq p <\infty\, .$$
\item Let $T_k$ denote the unit tangent vector for $\bndry\domain_k$ and $T$ denote the unit tangent vector of $\bndry\domain$. We have that

$$T_k\to  T \quad \sigma\text{-a.e.}\ \bndry\domain\ \  \text{and in every}\quad  L^p(\bndry\domain,\sigma)\ \ \text{for any}\ \  1\leq p <\infty\, .$$

\end{enumerate}
\end{lem}
\vskip0.1in

(Note that in conclusions {\em (b)} through  {\em (d)} the exponent $p=\infty$ cannot be allowed unless $\domain$ is of class $C^1$.)

\subsection{Boundary value problems for the Laplace operator}
Let $\domain$ be a simply connected Lipschitz domain. Given $F:\domain\to\mathbb{C}$, for any $\zeta\in\bndry\domain$ we define the \emph{non-tangential maximal function of $F$}
as
\begin{equation}\label{D:Fstar}
F^*(\zeta) := \sup_{ z\in \Gamma (\zeta)}|F(z)|
\end{equation}
where $\Gamma(\zeta)$ is the coordinate cone given in Definition \ref{D:ap}.
Furthermore,   define the \emph{non-tangential limit of $F$} as
\begin{equation}\label{D:Fdot}
\dot{F}(\zeta) := \lim_{\stackrel{z\to \zeta}{ z\in \Gamma (\zeta)}}F(z).
\end{equation}
It is known that if $F\in \Hs^p(\domain)$, then $\dot{F}(\zeta)$ exist for $\sigma$-a.e. $\zeta\in\bndry\domain$ and  $\dot F\in L^p(\bndry\domain, \sigma)$,   see \cite[Theorem 10.3]{Duren}.

\vskip0.12in

For convenience we  state without proof some relevant features of two well-known boundary value problems for the Laplace equation on Lipschitz domains.

\vskip0.1in

\begin{thm}
\cite[Def. 1.7.4; Coroll. 2.1.6 \& Thm 2.2.22]{Kenig-2};
\cite[Corollary 3.2]{V}
\label{T:DL} 
Let $\realdomain\subset \mathbb{R}^k, k\ge 2$, be a bounded simply connected Lipschitz domain. Consider the {\bf Dirichlet problem for Laplace's operator}
\begin{equation}\label{E:DL}
     \left\{
      \begin{array}{lcll}
      \Delta U &= &0 & \text{in} \ \ \realdomain;\\ \\
      \dot U(\zeta) &= &u(\zeta) & \text{for}\ \sigma\text{-a.e.}\ \zeta\in \bndry\realdomain;\\ \\
      U^*&\in& L^p(\bndry\realdomain, \sigma)
      \end{array}
\right.
\end{equation}
where the datum $u$ is in $L^p(\bndry\realdomain,\sigma)$. Then there is $\delta = \delta(\realdomain)>0$ such that \eqref{E:DL} is uniquely solvable whenever $u\in L^p(\bndry\realdomain, \sigma)$ with $2-\delta< p<\infty$. Morevoer
\begin{equation*}\label{E:NTbdDirichlet}
\| U^*\|_{L^p(\bndry\realdomain, \sigma)}\lesssim \| u\|_{L^p(\bndry\realdomain, \sigma)}.
\end{equation*}
\end{thm}
\vskip0.1in

Recall that the generalized normal derivative of $V$ at $\zeta\in\bndry\domain$ is
\begin{equation}\label{E:normal-deriv}
\frac{\dee V}{\dee n}(\zeta):=\lim_{\stackrel{z\to \zeta}{ z\in \Gamma (\zeta)}}\langle \nabla V (z), n(\zeta)\rangle_{\mathbb R} = \langle \dot{(\nabla V)}(\zeta), n(\zeta)\rangle_{\mathbb R}\quad \text{for}\quad \sigma-\text{a.e.}\ \zeta\in\bndry\realdomain,
\end{equation}
see \eqref{E:normal-deriv-hol}.

\begin{thm}\cite[(1.2) with $b:=0$]{LanSh};\cite[Corollary 2.1.11 p. 48, Theorem 2.2.22 p. 56]{Kenig-2};\cite[p.347]{DaKe}
\label{T:NL}
Let $\realdomain\subset \mathbb{R}^k, k\ge 2$ be a bounded simply connected Lipschitz domain. Consider the {\bf Neumann problem for Laplace's operator}
\begin{equation}\label{E:NL}
     \left\{
      \begin{array}{lcll}
      \Delta V & = & 0  &\text{in} \ \ \realdomain;\\ \\
      \displaystyle{\frac{\partial V}{\partial n}(\zeta)} &= &v(\zeta) &\ \sigma\text{-a.e.}\ \zeta\in \bndry\realdomain;\\ \\
     (\nabla V)^*&\in& L^p(\bndry\realdomain, \sigma)
      \end{array}
\right.
\end{equation}
where the datum $v$ is in $L^p(\bndry\real\domain, \sigma)$ and satisfies the compatibility condition
\begin{equation*}\label{E: compatibility Neumann}
\int\limits_{\bndry\domain} v(\zeta)\, d\sigma (\zeta) \ =\ 0.
\end{equation*}
Then there is $\n = \n (\realdomain)>0$ such that \eqref{E:NL} is uniquely solvable (modulo additive constants) whenever $v\in L^p(\bndry\realdomain, \sigma)$ with
   $1<p< 2 +\n$. 
  Moreover
\begin{equation*}\label{E:NTbdNeumann}
\| (\nabla V)^*\|_{L^p(\bndry\realdomain, \sigma)}\lesssim \| v\|_{L^p(\bndry\realdomain, \sigma)}.
\end{equation*}
where the implied constant depends solely on $\realdomain$.
\end{thm}

In each theorem the solution(s) admit an integral representation in terms of layer potential operators, which we omit for brevity.

\begin{rem}
The $p$-range 
in each theorem above is maximal in the Lipschitz category. See, for instance,  \cite[Remark 2.1.17]{Kenig-2}.
\end{rem}

%%%%%%%%%%%%%%%%%%%%%%%%%%%%%%%%%%%%%%%%%%%%%%%%%%%%%%%%%%%%%%%%%%%%%%%%%%%%%%%%%%%%%%%%%%%%%%%%%%%%%%%%%%%%%%%%%%%%%%%%%%%
\subsection{The Dirichlet problem for $\deebar$: proof of Theorem \ref{T:D-Dbar}}\label{DirichletSection}
For the sake of completeness we recall the proof of Cauchy's theorem for $h^1(\bndry\domain)$. Namely,

\begin{lem}\label{L:CauhyThmHp}
Let $\domain$ be a bounded Lipschitz domain. Then 
\begin{equation}\label{E:CTH1}
\int\limits_{\bndry\domain} f(\zeta)\, d\zeta = 0\qquad \text{for any}\quad  f\in h^1(\bndry\domain).
\end{equation}
\end{lem}
\begin{proof}
Let $\{\domain_k\}$ be a Ne\v{c}as exhaustion of $\domain$ and take $f\in h^1(\bndry\domain)$. By definition of $h^1(\bndry\domain)$, there exists $F\in\Hs^1(\domain)$ such that $f=\dot{F}$. For each $k$, $\domain_k$ is compactly contained in $\domain$ with smooth boundary and so by the classical Cauchy Theorem we have 
\[\int\limits_{\Lambda_k(\bndry\domain)} F(\zeta_k)d\zeta_k=0,\]
where $\Lambda_k$ is the diffeomorphism described in Lemma \ref{L:NecasExhaustion} part (a). Replacing $d\zeta_k$ by $T_k(\zeta_k)d\sigma_k(\zeta_k)$ in the above integral and making use of the change-of-variables formula Lemma \ref{L:NecasExhaustion} part (c),
  we have
$$\int\limits_{\bndry\domain}F(\Lambda_k(\zeta))T_k(\zeta)w_k(\zeta) d\sigma(\zeta) =  \int\limits_{\Lambda_k(\bndry\domain)} F(\zeta_k)T_k(\zeta_k) d\sigma_k(\zeta_k) = \int\limits_{\Lambda_k(\bndry\domain)} \!\!\! F(\zeta_k) d\zeta_k =0. $$
Passing $k\rightarrow \infty$ in the above and applying the dominated convergence theorem with the dominating function $MF^*\in L^{1}(\bndry\domain,\sigma)$ (here $M$ is a uniform bound on the $w_k$'s see Lemma \ref{L:NecasExhaustion}), we obtain (\ref{E:CTH1}).
\end{proof}

\begin{lem}\label{LaplacianLemma}
Let $\domain$ be a bounded simply connected Lipschitz domain and  $1<p<\infty$. If $F \in \mathcal H^p(\domain)$,
then
\begin{equation*}\label{T:Ddbar-hp:eq2_0}
\|F^*\|_{L^p(\bndry\domain,\sigma)} \approx \|\dot F\|_{L^p(\bndry\domain, \sigma)}.
\end{equation*}
\end{lem}
\begin{proof} The inequality: $\|F^*\|_{L^p(\bndry\domain,\sigma)} \geq \|\dot F\|_{L^p(\bndry\domain, \sigma)}$ is immediate from the definitions, thus we only need to show that $\|F^*\|_{L^p(\bndry\domain,\sigma)} \lesssim \|\dot F\|_{L^p(\bndry\domain, \sigma)}$ for any $1<p<\infty$.
\vskip0.1in
If $2-\delta(\domain)< p < \infty$ the desired inequality follows at once from the Dirichlet problem for harmonic functions (Theorem \ref{T:DL}) because $F\in \mathcal H^p(\domain)$ solves
\eqref{E:DL} with datum $u:= \dot{F}$. We next suppose that $1<p<2$ and we  let $G\in \vartheta(\domain)$ be any holomorphic antiderivative of $F$ (which exists since $F$ is holomorphic in $\domain$ which is simply connected); that is
$$
F(z) = G'(z),\quad z\in \domain.
$$
It follows that $G'\in \mathcal H^p(\domain)$ with
\begin{equation}\label{E:one}
|(\nabla G)^*(\zeta)|\approx  (G')^*(\zeta) = F^*(\zeta), \quad \text{and}\quad \dot{(G')}(\zeta) = F(\zeta)\quad \sigma-\text{a.e.}\ \zeta\in\bndry\domain\, .
\end{equation}
Hence
\begin{equation*}\label{E:cplx-not}
\frac{\partial G}{\partial n}(\zeta) = 
-i T(\zeta)\dot{F}(\zeta) \quad \text{for}\ \sigma - \text{a.e.}\ \zeta\in \bndry\domain,
\end{equation*}
see  \eqref{E:normal-deriv}.
Thus $G$ solves the Neumann problem for harmonic functions \eqref{E:NL}
with datum
$$
v(\zeta): = -iT(\zeta)\dot F(\zeta) \in L^p(\bndry\domain,\sigma)
$$
(note that such $v$ satisfies the compatibility condition:
$$
\int\limits_{bD}v (\zeta)d\sigma(\zeta) =0
$$
on account of Lemma \ref{L:CauhyThmHp}.)
Theorem \ref{T:NL}
 now ensures that   $\|(\nabla G)^*\|_p\lesssim \|\dot F\|_p$, which proves the desired inequality by way of \eqref{E:one}.

\end{proof}

\noindent{\em Proof of Theorem \ref{T:D-Dbar}}: 
Let $p>0$ and $f\in L^p(\bndry\domain, \sigma)$. 
We first show
that \eqref{E:D-Dbar} is solvable if, and only if, $f\in h^p(\bndry\domain)$. To this end,
suppose that $F$ is a solution of \eqref{E:D-Dbar}
with datum $f$:
then it is immediate from \eqref{E:D-Dbar}, 
\eqref{D:hp} and \eqref{D:Hp} that $F\in \Hs^p(\domain)$ and $f\in h^p(\bndry\domain)$. Conversely, if $f\in h^p(\bndry\domain)$, then $f = \dot F$ for some $F\in \Hs^p(\domain)$ and it follows from \eqref{D:Hp} that $F$ solves \eqref{E:D-Dbar}. Now
uniqueness for any $p>0$ follows from \cite[Theorem 10.3]{Duren}.
 Moreover, if $p\geq 1$ and $F\in \Hs^p(\domain)$ is the solution of \eqref{E:D-Dbar} with datum $f\in h^p(\bndry\domain)$, then $F = \mathbf{C}_Df$ by Cauchy Formula for $\Hs^1(\domain)$ \cite[Theorem 10.4]{Duren}).
Finally, if $1<p<\infty$ then  $\|F^*\|_p\lesssim \|f\|_p$ by Lemma \ref{LaplacianLemma}.
\qed

%%%%%%%%%%%%%%%%%%%%%%%%%%%%%%%%%%%%%%%%%%%%%%%%%%%%%%%%%%%%%%%%%%%%%%%%%%
\section{Holomorphic Sobolev-Hardy spaces}\label{SobolevHardySection}
Here we highlight the main features of $\Hs^{1,p}(\domain)$. It is worthwhile pointing out that
several of the results in this section do not require $\domain$ to be simply connected: indeed, 
the definitions of holomorphic Hardy space $\Hs^p(\domain)$ and of Sobolev-Hardy space $\Hs^{1,p}(\domain)$, see \eqref{D:Hp} and \eqref{D:H1q}, are meaningful for any bounded Lipshitz domain
 $\domain$ (whether or not simply connected). Furthermore, since Lipschitz domains are local epigraphs, any bounded Lipschitz domain must be finitely connected. Hence, an elementary localization argument shows that any $F\in\Hs^p(\domain)$ has a nontangential limit $\dot{F}$ defined $\sigma$-a.e. on $\bndry\domain$ which also lies in $L^p(\bndry\domain,\sigma)$ for any bounded Lipschitz domain $\domain$.

\begin{prop}\label{HSsubsetHardy}
Let $\domain\subset\mathbb C$ be a bounded Lipschitz domain and let $\alpha\in\domain$. Then
for any $G\in \mathcal H^{1, p}(\domain)$ and $1\leq p \leq \infty$, we have
  \begin{equation}\label{E:incl}
    G^*(\zeta) \lesssim  (G')^*(\zeta)+ \|\dot{(G')}\|_{L^p(bD,\sigma)} + |G(\alpha)|\quad \text{for any}\ \zeta\in\bndry\domain\, .
    \end{equation}
 \end{prop}

 We first prove the following   
\begin{lem}\label{Prop4Lemma}
Let $\domain$ be a bounded domain with smooth boundary. Then for any  $z, w\in \overline{\domain}$ there exists a piecewise smooth path $\gamma_w^z  \subset \overline{\domain}$ that joins $z$ and $w$, whose length does not exceed $\mathrm{diam}(\domain)+\sigma(\bndry\domain)$.
\end{lem}

\begin{proof}
For any two points $z,w\in\overline{\domain}$, let $\ell$ denote the oriented line segment starting at $z$ and ending at $w$. Construct a path $\gamma_w^z $ that is equal to the pieces of $\ell$ which are in
$\overline{\domain}$ and, for the pieces of $\ell$ which are outside of $\overline{\domain}$, $\gamma_w^z $ is the portion of $\bndry\domain$ connecting the point where $\ell$ leaves $\overline{\domain}$ to the point where $\ell$ re-enters $\overline{\domain}$. Then $\gamma_w^z $ has all the required properties.
\end{proof}

\begin{proof}[Proof of Proposition \ref{HSsubsetHardy}]
We claim that
for any $\zeta \in\bndry\domain$,
\begin{equation}\label{mm1}
G(z)\lesssim   (G')^*(\zeta)+ \|\dot{(G')}\|_{L^p(\bndry\domain, \sigma)} + |G(\alpha)| \quad \text{for all}\quad z\in \Gamma(\zeta)
\end{equation}
from which \eqref{E:incl} follows on account of the boundedness of $\domain$.
To prove \eqref{mm1}, let $\alpha\in\domain$ be fixed, and let  $\domain_\alpha$ be a member of a Ne\v{c}as exhaustion of $D$ such that $\alpha\in\domain_\alpha$. In particular
$$\domain_\alpha\  \ \text{is smooth, and}\quad  \domain_\alpha \cap \Gamma(\zeta)\ne \emptyset\quad \text{for all}\quad  \zeta \in\bndry\domain.$$ 

Let $\zeta\in\bndry\domain$ and   $z\in \Gamma (\zeta)$. Then  
  we have
\begin{equation*}\label{line}
|G(z)| \le  \int\limits_{\ell_w^z}|G' (\mu)|\ d\sigma(\mu) +|G(w)|
\quad \text{ for any}\ w\in \domain_\alpha\cap \Gamma(\zeta),
\end{equation*}
where $\ell_w^z$ is the line segment joining $w$ to $z$. The above line integral makes sense since  $\ell_w^z\subset \Gamma(\zeta)\subset \domain$  by convexity of the cone $\Gamma(\zeta)$ (see Definition \ref{D:ap}) and $G'\in \vartheta(\domain)$. Moreover, the length $L_1$ of $\ell_w^z $  does not exceed the diameter of $\domain$. Hence
 \begin{equation}\label{E:temp1}
 |G(z)| \le
 L_1\, (G')^{\!*}(\zeta) + |G(w)|\quad \text{for any}\quad z\in \Gamma (\zeta)\quad \text{and}\quad w\in \domain_\alpha\cap \Gamma(\zeta).
 \end{equation}

Next since $\domain_\alpha$ is smooth, for each $w\in \domain_\alpha$ 
by Lemma \ref{Prop4Lemma} there is a path $\gamma_\alpha^w\subset \domain_\alpha$ joining the base point $\alpha$ to $w$, whose length $L_2$ does not exceed some constant dependent only on $\domain$. 
The Cauchy integral formula for $G'$ and H\"older inequality give 
\begin{equation*}
|G(w) - G(\alpha)| \leq
\int\limits_{\gamma_\alpha^w} | G' (\mu)|\, d\sigma(\mu) \leq
\frac{L_2}{2\pi} \sup_{\mu\in \gamma_\alpha^w}\int\limits_{\bndry\domain} \frac{|\dot{(G')}(\eta)|}{|\eta-\mu|}d\sigma(\eta) \leq
\frac{L_2|\bndry\domain|^{1-\frac{1}{p}}}{2\pi \dist(\bndry\domain, \domain_\alpha)}
\|\dot{(G')}\|_{L^p(\bndry\domain,\sigma)},
\end{equation*}
for $1\leq p\leq\infty$, whence
\begin{equation}\label{hh1}
    |G(w)|\lesssim |G(\alpha)| + \|\dot{(G')}\|_{L^p(\bndry\domain,\sigma)} \quad \text{for any}\quad w\in \domain_\alpha\quad \text{and}\quad 1\leq p\leq\infty.
\end{equation}
The conclusion follows from \eqref{E:temp1} and \eqref{hh1}.
\end{proof}

Note that inequality \eqref{mm1} is meaningful also for $G\in \mathcal H^{1, p}(\domain)$ with $0<p<1$, but for $p$ in such range we are unable to prove it.

\begin{cor}\label{HSsubsetHardyC}
 Let $\domain\subset\mathbb C$ be a bounded Lipschitz domain and let $\alpha\in\domain$. For any $G\in \mathcal H^{1, p}(\domain)$ and $1 \leq p \leq \infty$ we have the set inclusion
\begin{equation}\label{E:H1q-in=Hq}
    \mathcal H^{1, p}(\domain)\subseteq \Hs^p(\domain),\quad   1 \leq p \leq \infty.
    \end{equation}   
   If in addition $\domain$ is simply connected and $1<p<\infty$, then
      \begin{equation}\label{E:NTM-bound1}
    \| G^*\|_{L^p(\bndry\domain, \sigma)} \lesssim |G(\alpha)| + \|G'\|_{\mathcal H^{p}(\domain)}.
    \end{equation}
 \end{cor}    

\begin{proof} The definition of $G\in \mathcal H^{1, p}(\domain)$
gives that $G\in \vartheta(\domain)$ and $G'\in \mathcal H^p(\domain)$. Consequently $(G')^*, \dot{(G')}\in L^p(\bndry\domain, \sigma)$. By the pointwise estimate \eqref{E:incl} and the boundedness of $\domain$, we get $G^*\in L^p(\bndry\domain, \sigma)$ and thus $G\in \Hs^p(\domain)$ for any $p\geq 1$, proving \eqref{E:H1q-in=Hq}.

If moreover, $1<p<\infty$, then by Lemma \ref{LaplacianLemma} (which requires $\domain$ to be simply connected) we also have
$$  \| (G')^*\|_{L^p(\bndry\domain, \sigma)} \lesssim  \| \dot{(G')}\|_{L^p(\bndry\domain, \sigma)} = \|G'\|_{\mathcal H^{p}(\domain)}. $$
  Together with   \eqref{E:incl}  this gives
the desired inequality \eqref{E:NTM-bound1}.
\end{proof}
 
It follows from \cite[Theorem 10.3]{Duren} and \eqref{E:H1q-in=Hq}  
that  every element of $\Hs^{1, p}(\domain)$ has a non-tangential limit     
 in   $h^{1, p}(\bndry\domain)$ when $\domain$ is a bounded simply connected Lipschitz domain, see \eqref{D:h1q}. The aforementioned localization argument and  \eqref{E:H1q-in=Hq} give that this is true also for multiply-connected $\domain$. 
  
\vskip0.1in

In view of Corollary \ref{HSsubsetHardyC}, 
for any $\alpha\in \domain$, any $G\in \mathcal H^{1, p}(\domain)$  and any $1 \leq p \leq \infty$ we set
\begin{equation*}\label{E:normSHdef1}
\|G\|_{\mathcal H^{1, p}(\domain)} := |G(\alpha)|+ \|G'\|_{\mathcal H^p(\domain)}.
\end{equation*}
 It is easy to  see that the above defines a family of norms for $\mathcal H^{1, p}(\domain)$ (one norm for each choice of the base point $\alpha\in\domain$). The proof of Proposition \ref{HSsubsetHardy} 
also shows that 
all such norms are equivalent to one other. Indeed, given $\alpha,\tilde{\alpha}\in \domain$ choose $k\in\mathbb N$ so that $\alpha,\tilde{\alpha} \in D_{k}$ (where $D_k$ is a member of a Ne\v{c}as exhaustion for $\domain$), then the proof leading up to \eqref{hh1} shows that  $|G(\alpha)-G(\tilde{\alpha})|\lesssim \|\dot{(G')}\|_{L^p(\bndry\domain,\sigma)}= \|G'\|_{\mathcal H^p(\domain)}$. As such  we shall not specify $\alpha $ in the notation of $ \|\cdot\|_{\mathcal H^{1, p}(\domain)}  $.

\begin{lem}\label{4i} 
Given $\domain$   a bounded simply connected Lipschitz domain and 
$1\leq p \leq \infty$, $\mathcal H^{1, p}(\domain)$ is a Banach space with the norm $\|\cdot\|_{\mathcal H^{1, p}(\domain)} $. 
\end{lem}

\begin{proof}
 Let $G_n$ be a Cauchy sequence in $\mathcal H^{1, p}(\domain)$. Then $G_n'$ is a Cauchy sequence in $\Hs^p(\domain)$ and thus it converges to a function $F$ in $\Hs^p(\domain)$. Since $\domain$ is simply connected, $F$ has a holomorphic antiderivative $G$. Furthermore, we can choose an antiderivative $G$ such that $G(\alpha)=\lim_{n\to\infty} G_n(\alpha)$. Then $G\in \mathcal H^{1, p}(\domain)$ and $G_n$ converges to $G$ in the $\mathcal H^{1, p}(\domain)$-norm. 
 \end{proof}

 \begin{rem}
 The space $\mathcal H^{1, p}(\domain)$  is not closed in the  $\Hs^{p}(\domain)$-norm, already for $D = \mathbb D$ (the unit disc). To see this, let $F\in \Hs^p(\mathbb{D})$ be such that $F'\notin\Hs^p(\mathbb{D})$ (any $F\in\Hs^p(\mathbb{D})$ that does not continuously extend to $\overline{\mathbb{D}}$ 
  will satisfy this condition; see, for example, \cite[Theorem 3.11]{Duren}). 
 Let $F_n$ be the $n^{th}$ Taylor polynomial of $F$ centered at 0: clearly $F_n\in\Hs^{1,p}(\mathbb{D})$ ($F_n$ is a holomorphic polynomial). Moreover, since $F\in \Hs^p(\mathbb{D})$, $F_n$ converges to $F$ in the $\Hs^{p}(\mathbb D)$-norm.  So
$\displaystyle{F\in\overline{\Hs^{1,p}(\mathbb{D})}^{\Hs^p}\setminus \Hs^{1,p}(\mathbb{D})}$,
proving that $\Hs^{1,p}(\mathbb{D})$ is not closed in the $\Hs^p$-norm.
 \end{rem}

\vskip0.12in
Let $W^{1,p}(\bndry\domain, \sigma)$, $p\geq1$, be the {\bf Sobolev space} consisting of all $L^p(\bndry\domain, \sigma)$ functions whose first order  derivatives along the tangential direction of $\bndry\domain$ exist in the distributional sense and are in $L^p(\bndry\domain, \sigma)$. Here we shall adopt the  definition of the tangential derivative in \cite{V} that is invariant under rotations and translations. To be precise, given a coordinate rectangle $R=(a,b)\times(c,d)$  and Lipschitz boundary function $\phi$ for $\domain$ as in  Definition \ref{de} with angle $\theta=0$ (for simplicity and without loss of generality), 
\begin{itemize}
\item \cite[Definition 1.7]{V} we say that a {\em real-valued}  function $u\in W^{1,p}(\bndry\domain, \sigma)$ if $u\in L^p(\bndry\domain, \sigma)$, and  there exists $f\in L^p(R\cap\bndry\domain,d\sigma)$ such that
\begin{equation}\label{SobolevDefEq1}
-\int_{a}^b u\left( t, \phi(t)\right)\psi'(t)dt = \int_a^b  f\left( t, \phi(t) \right)\psi(t)dt \mbox{ for all } \psi\in C^\infty_c((a,b));
\end{equation}
\item \cite[Definition 1.9]{V} the {\em tangential derivative} $ \dee_T u(\zeta) $ of $u$ at $\sigma$-a.e.
 $\zeta :=   (t, \phi (t) )\in \bndry\domain$ is defined by
$$
    \dee_T u(\zeta) T(\zeta):  =    (f(\zeta), 0) - \langle ( f(\zeta), 0),  n(\zeta)\rangle_{\mathbb R} n(\zeta) = \langle  ( f(\zeta), 0), T(\zeta)\rangle_{\mathbb R}T(\zeta)
    $$
where $n(\zeta)$ is the outer unit normal vector at $\zeta\in\bndry\domain$. Thus,
$$
    \dee_T u(\zeta) :  =    \langle  ( f(\zeta), 0), T(\zeta)\rangle_{\mathbb R}.
    $$
  In our local coordinates the unit tangent vector at $\zeta := (t, \phi (t))$ is
$$T(\zeta) =\frac{  \left( 1, \phi' (t)\right)}{|\left(1,\phi' (t)\right)|},$$
hence
\begin{equation}\label{td}\dee_T u(\zeta) = \frac{ f(\zeta)}{|\left(1,\phi' (t)\right)|},\quad \zeta:=  \left(t, \phi (t)\right),\quad t\in (a, b). \end{equation}

\item A {\em complex-valued} function $u+iv$ is said to be in $ W^{1,p}(\bndry\domain, \sigma)$ if both its real part $u$ and its imaginary part $v$ are in $ W^{1,p}(\bndry\domain, \sigma)$, in which case we let $$ \dee_T (u+iv) (\zeta) = \dee_T u(\zeta) + i \dee_T v(\zeta).$$

\item $W^{1,p}(\bndry\domain, \sigma)$ is a Banach Space with
$$
\|f\|_{W^{1,p}(\bndry\domain, \sigma)} := \|f\|_{L^p(\bndry\domain, \sigma)} + \|\dee_T f\|_{L^p(\bndry\domain, \sigma)} 
$$
\end{itemize}
\medskip

\begin{prop}\label{d} Let $\domain$ be a bounded simply connected Lipschitz domain. The following properties hold:
\begin{enumerate}[(a).]
\item\label{dparta}  $h^{1, p}(\bndry\domain)\subset W^{1, p}(\bndry\domain, \sigma)$ for any $p\geq 1$. That is, for any $G\in \mathcal H^{1, p}(D)$, the tangential derivative $\dee_T\dot G$ is a regular distribution and satisfies
 \begin{equation}\label{E:tan-der}
 \dee_T\dot G (\zeta) = T(\zeta)\dot{(G')}(\zeta)\quad \sigma\text{-a.e.}\ \zeta\in\bndry\domain.
 \end{equation}

\item\label{dpartb} $h^{1, p}(\bndry\domain)\subset C^{1-\frac{1}{p} }(\bndry\domain)$ for any $1\le p<\infty $. More precisely, for any $G\in \mathcal H^{1, p}(D)$ there is a unique $g\in C^{1-\frac{1}{p}} (\bndry\domain)$ such that $\dot{G}(\zeta) = g(\zeta)$ for $\sigma$-a.e. $\zeta\in \bndry\domain$.
\end{enumerate}
\end{prop}
Recall that for a closed set $X$ we say that $g\in C^{1-\frac{1}{p}}(X)$ if its H\"{o}lder norm, defined by
\[\|g\|_{C^{1-\frac{1}{p}}(X)}:= \sup_{x\in X}|g(x)|+ \sup_{x,y\in X}\frac{|g(x)-g(y)|}{|x-y|^{1-\frac{1}{p}}},\]
is finite. When $p=1$, it reduces to the continuous function space $C(X)$.

\begin{proof}[Proof of Proposition \ref{d}:]
For part $(a)$, 
let    $\{R_j\}_{j=1}^m$ be a family of coordinate rectangles covering $\bndry\domain$
 as in Lemma \ref{L:NecasExhaustion} part $(b)$; for any member of such family we henceforth omit the label $j$ and write $R =(a, b)\times (c, d)$ as well as  $\theta, \phi, \phi_k$ as in Definition \ref{de} and Lemma \ref{L:NecasExhaustion}. 
 For any  testing function $\psi\in C_c^\infty((a, b))$, we have   
 \begin{equation*}
     \begin{split}
        -\int\limits_{{a}}^{b} \dot{G}\left( t, \phi(t)\right) \psi'(t) dt
         =&  -\lim_{k\rightarrow \infty } \int\limits_{a}^{b} G  \left(  t, \phi_k(t)\right) \psi'(t) dt \\
         =&  \lim_{k\rightarrow \infty } \int\limits_{a}^{b}\left( G_x \left(  t, \phi_k(t)\right) + G_y\left(  t, \phi_k(t)\right)     \left(\phi_k\right)'(t)\right) \psi(t) dt
\end{split}
 \end{equation*}
Here the first equality is due to Lebesgue's dominated convergence theorem (where we use that $G^*\in L^p(\bndry\domain, \sigma ) \subseteq L^1(\bndry\domain, \sigma )$); the second equality follows from integration by parts.  Since $G$ is holomorphic in $\domain$, by the Cauchy-Riemann equation for holomorphic functions, we have for $t\in (a, b)$, 
 $$ G_x \left(  t, \phi_k(t)\right) = G'\left(  t, \phi_k(t)\right)  \ \ \text{and}\ \ G_y \left(  t, \phi_k(t)\right) = iG'\left(  t, \phi_k(t)\right). $$
 Thus 
       \begin{equation*}
     \begin{split}  
     -\int\limits_{{a}}^{b} \dot{G}\left( t, \phi(t)\right) \psi'(t) dt    =&  \lim_{k\rightarrow \infty } \int\limits_{a}^{b} G'\left(  t, \phi_k(t)\right)\left(1+i \left(\phi_k\right)'(t)\right)    \psi(t) dt\\
     =&   \int\limits_{a}^{b}   \dot{(G')} \left(t, \phi (t)\right) 
         \left(1 +i \phi'(t) \right)\psi(t)  dt, 
     \end{split}
 \end{equation*}
where the last equality  is due to   Lemma \ref{L:NecasExhaustion} parts $(b)$  and  again Lebesgue's dominated convergence theorem (where this time we use that $(G')^*\in L^p(\bndry\domain, \sigma ) \subseteq L^1(\bndry\domain, \sigma )$). 

Hence  \eqref{SobolevDefEq1} holds with 
$$f( t, \phi(t)):= \dot{(G')}( t, \phi(t))   \left(1 +i \phi'(t) \right).$$
By \eqref{td} and the fact that in complex notation we have that $T =\frac{  1 +i\phi' }{|1+i\phi'|}$ on $R\cap \bndry\domain$,  it follows that for $\sigma$-a.e. $\zeta=(t+i\phi(t))\in \bndry\domain$,
$$\dee_T \dot G(\zeta) = \frac{\dot{(G')}(\zeta)  \left(1 +i\phi'(t) \right)}{|1+i\phi'(t)|} = T(\zeta) \dot{(G')}(\zeta). $$
Conclusion $(a)$ is thus proved.
\vskip0.2in

For the proof of conclusion $(b)$, we let $\{\domain_k\}$ be a 
Ne\v{c}as exhaustion $\{\domain_k\}$ of $\domain$, see Lemma \ref{L:NecasExhaustion}. Given $G\in\Hs^{1,p}(\domain)$ we fix a point $\nu_0\in\bndry\domain$ where $\dot{G}(\nu_0)$ exists (since $G\in \Hs^{1,p}(\domain)\subseteq\Hs^p(\domain)$, $\dot{G}$ exists $\sigma$-a.e.). Let $\zeta\in\bndry\domain$ be arbitrary and let $\gamma(\nu_0,\zeta)$ be the arc in $\bndry\domain$ from $\nu_0$ to $\zeta$ oriented in the positive direction, see \eqref{E:path-def}. Since $G$ is holomorphic on $\domain$ and $\Lambda_k(\gamma(\nu_0,\zeta))$ is a smooth curve contained in $\domain$, by the Fundamental Theorem of Calculus for line integrals we have
\begin{equation}\label{E:temp}
\begin{split}
G(\Lambda_k(\zeta))&= G(\Lambda_k(\nu_0)) +\!\!\! \int\limits_{\Lambda_k(\gamma(\nu_0,\zeta))} 
\!\!\!\!\!\nabla_{T_k} \,G(\eta_k)\, d\sigma_k(\eta_k)\\ 
& =  G(\Lambda_k(\nu_0)) +\!\!\! \int\limits_{\Lambda_k(\gamma(\nu_0,\zeta))} \!\!\!\!\!T_k(\eta_k)\, G'(\eta_k)\, d\sigma_k(\eta_k) \\ 
&=  G(\Lambda_k(\nu_0)) + \int\limits_{\gamma(\nu_0,\zeta)} (T_k\circ\Lambda_k)(\eta) \,(G'\circ\Lambda_k)(\eta)\, w_k(\eta)\, d\sigma (\eta),
\end{split}
\end{equation}
where the last identity is due to
Lemma \ref{L:NecasExhaustion}. Now for any $\zeta\in\bndry\domain$ the right-hand side converges as $k\to\infty$ by the Lebesgue's dominated convergence theorem (again as $(G')^*\in L^p(\bndry\domain, \sigma ) \subseteq L^1(\bndry\domain, \sigma )$), and by Lemma \ref{L:NecasExhaustion}, to
$$
g(\zeta)\ :=\ \dot{G}(\nu_0) + \int\limits_{\gamma(\nu_0,\zeta)}\!\! T(\eta) \,\dot{(G')}(\eta)\, d\sigma (\eta).
$$
It follows from the above that for any $\zeta\in\bndry\domain$ the left hand side of \eqref{E:temp} also converges as $k\to \infty$, and in fact
$$
g(\zeta)=\lim\limits_{k\to\infty} G(\Lambda_k(\zeta)) = \dot{G}(\zeta)
\mbox{ for $\sigma$-a.e. $\zeta\in\bndry\domain$,}$$
because $\Lambda_k(\zeta)\in \Gamma(\zeta)$, see Lemma \ref{L:NecasExhaustion}. 
Now H\"older's inequality gives
$$|g(\zeta)-g(\xi)|
 \leq \|\dot{(G')}\|_{L^p(\bndry\domain,\sigma)}\left(\sigma(\gamma_{\xi}^\zeta) \right)^{\!\!{1-\frac{1}{p}}}\,.
 $$
 But
 $$
 \sigma(\gamma_{\xi}^\zeta)
  \lesssim
  |\zeta-\xi|\quad \text{for any}\ \zeta, \xi\in\bndry\domain$$
since Lipschitz domains are chord-arc, and  the implied constant depends only on $\domain$  (see, for example, Section 7.4 in \cite{Pom}).
 Hence $g\in C^{1-\frac{1}{p}}(\bndry\domain)$; the proof is concluded.
\end{proof}

\begin{rem}\label{mor}
 In the case when $1<p<\infty$, Proposition \ref{d} $(b)$ can be directly obtained by incorporating Proposition \ref{d} {\it (a)} with the Sobolev embedding theorem (\cite[Theorem 3.6.6]{Mor}). In fact, as a result of this, for any $g\in h^{1, p}(\bndry\domain)$, one has $g\in W^{1, p}(\bndry\domain, \sigma) \subset C^{1-\frac{1}{p}}(b\domain)$ along with the following estimate
 $$\|g\|_{C^{1-\frac{1}{p}}(b\domain) } \lesssim \|g\|_{W^{1, p}(\bndry\domain, \sigma) }.$$
\end{rem}

\section{The Regularity problem for $\deebar$: proof of Theorem \ref{T:D-Dbar-reg}}
\label{S:T:D-Dbar-reg}

Let $f\in W^{1,p}(\bndry\domain, \sigma)$. 
We first show
that \eqref{E:D-Dbar-reg} is solvable if, and only if, $f\in h^{1,p}(\bndry\domain)$. To this end,
suppose that $F$ is a solution of \eqref{E:D-Dbar-reg}
with datum $f$: then $F$ is holomorphic, and it is immediate from \eqref{E:D-Dbar-reg}, 
    \eqref{D:Hp} and \eqref{D:H1q} that $F'\in \Hs^p(\domain)$ which means that $F\in\Hs^{1,p}(\domain)$, and therefore $f\in h^{1,p}(\bndry\domain)$, see \eqref{D:h1q}. Conversely, if $f\in h^{1,p}(\bndry\domain)$, then $f = \dot F$ for some $F\in \Hs^{1,p}(\domain)$  and \eqref{D:H1q} gives that $F$ solves \eqref{E:D-Dbar-reg}. Uniqueness follows from Theorem \ref{T:D-Dbar} (since $\Hs^{1,p}(\domain)\subset\Hs^p(\domain)$).

 If $p> 1$ and $F\in \Hs^{1,p}(\domain)$ is the solution of \eqref{E:D-Dbar-reg} with datum $f\in h^{1,p}(\bndry\domain)$, then  $F=\mathbf{C}_D f$ by Theorem \ref{T:D-Dbar}. 
Make use of Remark \ref{mor} and the boundedness of the operator $\mathbf{C}_D $ from $C^{1-\frac{1}{p}}(b\domain) $ to $C^{1-\frac{1}{p}}(\overline \domain) $ (see  \cite[Theorem 3.3]{Mclean}). Namely, the holomorphic function $\mathbf{C}_D f$ extends to $F\in C^{1-\frac{1}{p}}(\overline \domain)$ with
 \begin{equation*}\label{E:temp-5}
 \|F\|_{C^{1-\frac{1}{p}}( \overline\domain)}\lesssim \|f\|_{C^{1-\frac{1}{p}}(\bndry\domain)}\lesssim \|f\|_{W^{1, p}(\bndry\domain, \sigma)}.
 \end{equation*}
  In particular, 
 $$ \|F^*\|_{L^\infty(\bndry\domain, \sigma)}\lesssim \|f\|_{W^{1, p}(\bndry\domain, \sigma)}. $$
 
 On the other hand, we also have that
 $F'$ solves \eqref{E:D-Dbar} with datum 
 $$\overline{T}(\zeta)\,\dee_T f(\zeta)$$
 which belongs to $h^p(\bndry\domain)$ by Proposition \ref{d}$(a)$, hence
 \begin{equation*}\label{E:temp-6}
 \|(F')^*\|_{L^p(\bndry\domain, \sigma)}\ \lesssim\ \|\overline{T}\,\partial_Tf\|_{L^p(\bndry\domain, \sigma)} \ =\ \|\partial_Tf\|_{L^p(\bndry\domain, \sigma)}
 \end{equation*}
again by Theorem \ref{T:D-Dbar}.
\qed

\section{The Neumann problem for $\deebar$: proof of Theorem \ref{T:NDbar}}\label{NeumannSection}

Let $p\geq 1$ and $g\in L^p(\bndry\domain, \sigma)$.
We first show that the holomorphic Neumann problem \eqref{NdbarIntro}
 is solvable if and only if $g\in\neu^{p}(\bndry\domain)$. To see this, suppose that $G$ is a solution to \eqref{NdbarIntro} with datum $g$: then it is immediate from \eqref{NdbarIntro} and \eqref{D:H1q}
 that  $G\in \Hs^{1, p}(\domain)$, hence $\dot{G}\in h^{1,p}(\bndry\domain)$. The latter together with 
 \eqref{NdbarIntro} again, \eqref{E:normal-deriv-hol},  \eqref{E:tan-der} and \eqref{D:Ndata} give that
 $g\in\neu^{p}(\bndry\domain)$. Conversely, if $g\in\neu^{p}(\bndry\domain)$ then $g = -i\partial_T\dot{G}$ for some $G\in \Hs^{1, p}(\domain)$ and it follows from 
 \eqref{E:normal-deriv-hol}, \eqref{E:tan-der}, \eqref{D:H1q} and \eqref{D:Hp}
  that
 $G$ solves \eqref{NdbarIntro} with datum $g$. Now \cite[Theorem 10.3]{Duren}
   gives that $G'$ is unique, hence $G$ is unique modulo additive constants. Next we note that if $G\in \Hs^{1,p}(\domain)$ is a solution of \eqref{NdbarIntro} with datum $g\in \neu^p(\bndry\domain)$, then $F:= G'\in \Hs^p(\domain)$ is the solution of the holomorphic Dirichlet problem \eqref{E:D-Dbar} with datum $$f(\zeta):=\, 
 i \overline{T}(\zeta) g(\zeta)\, =\, i \overline{T}(\zeta)\frac{\dee G}{\dee n}(\zeta)\, =\, \dot{(G')}(\zeta)$$
 by \eqref{NdbarIntro} and \eqref{E:normal-deriv-hol}. Hence $f\in h^p(\bndry\domain)$ and Theorem \ref{T:D-Dbar} gives
$G' = \mathbf{C}_D (i\overline{T} g)$.

Theorem \ref{T:D-Dbar} also grants that
 \begin{equation*} 
 \|(G')^*\|_{L^p(\bndry\domain, \sigma)}\lesssim 
  \|i\overline{ T}g\|_{L^p(\bndry\domain, \sigma)}=\|g\|_{L^p(\bndry\domain, \sigma)}\qquad 1<p<\infty.
 \end{equation*}
 Furthermore, it is clear from all we did that for any $\alpha\in \domain$, the holomorphic Neumann problem \eqref{NdbarIntro} has a unique solution 
 $G_\alpha\in \Hs^{1,p}_\alpha(\domain)$. 

For the remainder of the proof, we will assume that $1<p<\infty$. Fix one point $\zeta_0\in \bndry\domain$ and define
$$h_0(\zeta) =  \int\limits_{\gamma(\zeta_0, \zeta)}\!\!\!  ig(\eta)d\sigma(\eta) \quad 
\text{for every}\quad 
\zeta\in \bndry\domain.  $$
where $\gamma(\zeta_0, \zeta)$ is as in \eqref{E:path-def}.
Let $\xi\in\bndry\domain$ be a Lebesgue point of $g$. Since $g\in L^p(\bndry\domain,\sigma)\subseteq L^1(\bndry\domain,\sigma)$, Lebesgue points exist $\sigma$-a.e. on $\bndry\domain$. We have for any $\zeta\in\bndry\domain$,
\begin{equation*}
\left|\frac{h_0(\zeta)-h_0(\xi)}{\int_{\gamma(\xi, \zeta)} d\sigma} - ig(\xi)\right| = \left|\frac{\int_{\gamma(\xi, \zeta)} \left(g(\eta)-g(\xi)\right)d\sigma(\eta)}{\int_{\gamma(\xi,\zeta)} d\sigma}\right|.
\end{equation*}
The right-hand side of the above equation tends to $0$ as $\zeta$ approaches $\xi$ along $\bndry\domain$ since $\xi$ is a Lebesgue point of $g$. Thus,
$h_0$ admits a tangential derivative $\sigma$-a.e. and moreover
\begin{equation}\label{neum}
    \partial_T h_0 =ig\, .
\end{equation}
H\" older's inequality gives
$ \|h_0\|_\infty \lesssim \|g\|_p$;
it follows that
 $h_0 \in W^{1, p}(\bndry\domain, \sigma)$
with 
 $$\left\|h_0\right\|_{W^{1, p}(\bndry\domain, \sigma) } \lesssim \|g\|_{L^p(\bndry\domain, \sigma) }  .$$
Furthermore, since $g\in \neu^{p}$, we also have $h_0\in h^{1, p}(\bndry\domain)$ by \eqref{neum} and the definition of $\neu^{p} $, along with the fact that every function in $  W^{1, p}(\bndry\domain, \sigma)$ with vanishing tangential derivative is constant. See, for instance, \cite[Theorem 3.6.5]{Mor}. Now \cite[Theorem 3.3]{Mclean} together with the Sobolev embedding $W^{1, p}(\bndry\domain, \sigma)\subset C^{1-\frac{1}{p}}(\bndry\domain)$ in Remark \ref{mor},
 gives that $\mathbf C_D\left(h_0 \right)\in C^{1-\frac{1}{p} }( \overline{\domain})$ with 
 \begin{equation}\label{hol}
     \|\mathbf C_D\left(h_0 \right)\|_{C^{1-\frac{1}{p} }( \overline{\domain})}  \lesssim \|h_0\|_{C^{1-\frac{1}{p} }(\bndry\domain)}  \lesssim \|h_0\|_{W^{1, p}(\bndry\domain, \sigma) } \lesssim \|g\|_{L^p(\bndry\domain, \sigma) }. 
 \end{equation}
 
 Let $G_\alpha$ be the unique solution to \eqref{NdbarIntro} with datum $g$, hence $G_\alpha (\alpha) =0$ and $ \partial_T\dot{G_\alpha} =ig$   by \eqref{E:normal-deriv-hol} and \eqref{E:tan-der}, and \eqref{neum} along with \cite[Theorem 3.6.5]{Mor} give that $ \dot G_\alpha = h_0 + \lambda_\alpha$ for some constant $\lambda_\alpha$. 
 That is, $G_\alpha$ solves the Regularity problem \eqref{E:D-Dbar-reg} with datum
  $$h_0+\lambda_\alpha \in h^{1, p}(\bndry\domain).$$
  (Note that $\lambda_\alpha\in h^{1, p}(\bndry\domain)$ since $H(z):= \lambda_\alpha z\in \mathcal H^{1, p}(\domain)$, see \eqref{D:h1q} and \eqref{D:Hp}.) 
 \vskip0.1in
 
 Theorem \ref{T:D-Dbar-reg} now gives
 \begin{equation}\label{E:temp10}
 G_\alpha(z)  =\mathbf C_\domain h_0 (z) + \lambda_\alpha,\quad z\in \domain
 \end{equation} 
 (recall that  $\mathbf C_\domain(1)(z) \equiv 1$). Since $G_\alpha(\alpha) =0$, we further have $\lambda_\alpha = -\mathbf C_\domain h_0 (\alpha)$, hence $G_\alpha$ admits the representation \eqref{E:temp11}. 
 In particular,  \eqref{hol} gives that 
     $\lambda_\alpha\lesssim  \|g\|_{L^p(\bndry\domain, \sigma)}$ and that
  $G_\alpha\in C^{1-\frac{1}{p} }(\overline{\domain})  $ with
\begin{equation*}\label{E:bound-alpha-one}
 \|G_\alpha\|_{C^{1-\frac{1}{p} }(\overline{\domain})}\ \lesssim\  \|g\|_{L^p(\bndry\domain, \sigma)}.
 \end{equation*}
 In particular,
\begin{equation*}\label{E:bound-alpha}
\|G_\alpha^*\|_{L^\infty(\bndry\domain, \sigma)}\ \lesssim\  \|g\|_{L^p(\bndry\domain, \sigma)}.  
\end{equation*}
 \qed

\begin{rem}
The above construction does not depend on the choice of $\zeta_0$ in the following sense. If another point $\zeta_1$ were chosen instead with a corresponding function $h_1$, then $h_0$ and $h_1$ would only differ by an additive constant that can be chosen so that
$h_0(\zeta) - h_0(\alpha)=h_1(\zeta) - h_1(\alpha)$ for $\sigma$-a.e. $\zeta\in\bndry\domain$. 
\end{rem}

%%%%%%%%%%%%%%%%%%%%%%%%%%%%%%%%%%%%%%%%%%%%%%%%%%%%%%%%%%%%%%%
%%%%%%%%%%%%%%%%%%%%%%%%%%%%%%%%%%%%%%%%%%%%%%%%%%%%%%%%%%%%%%%%%
\section{The Robin problem for $\deebar$}\label{RobinSection}
We begin with an example that
shows that the uniqueness of the holomorphic Robin problem fails in general if the compatibility condition \eqref{E:compint} is dropped.

\begin{example}\label{rr}
Consider the following homogeneous holomorphic Robin problem on $\mathbb D$:
\begin{equation}\label{Rdbare}
     \left\{
      \begin{array}{lcll}
      \bar\partial G &= &0 & \text{in} \ \ \mathbb D;\\ \\
                        \displaystyle\frac{\partial G}{\partial n}(\zeta) - \dot G(\zeta)&= &0 & \text{for}\  \ \zeta\in \bndry\mathbb D\\\\
             (\nabla G)^*\in L^p(\bndry\mathbb D)
      \end{array}
\right.
\end{equation}
Making use of the fact that $ T(\zeta) = i\zeta $ on $b\mathbb D$ 
one can directly verify that  $G(z): =Cz$ solves \eqref{Rdbare} for any constant $C$. Note that $b=-1$ on $b\mathbb D$ with
$$  \int\limits_{b\mathbb D} b(\zeta) d\sigma(\zeta) =-2\pi. $$
\end{example}

\subsection{The smoothing operator $\mathcal T_b$}

For $\zeta, \xi\in \bndry\domain$ with $\zeta\neq \xi$, recall that  $\gamma (\zeta, \xi)$ is  the piece of $\bndry\domain$   joining $\zeta$ to $\xi$ in the positive direction. In particular, we define $\gamma(\zeta, \zeta) =\emptyset$; if $\xi$ approaches $\zeta$ along the positive orientation of $\gamma$,  we denote it by  $\xi\rightarrow \zeta^-$, with  $\gamma(\zeta, \zeta^-) = \bndry\domain$.

\begin{prop}\label{P:Top} Let $\domain$ be a bounded simply connected Lipschitz domain and let $1 < p<\infty$. Assume $b\in L^p(\bndry\domain,\sigma)$ satisfies the compatibility condition \eqref{E:compint}. Then  for the operator $\mathcal T_b$   defined in \eqref{E:auxOp}, we have that
\begin{itemize}
\item[\tt(i.)]\quad $\mathcal T_b$ is bounded: $ L^p(\bndry\domain, \sigma) \to C^{1-\frac{1}{p}}(\bndry\domain)$. Namely
\begin{equation*}\label{E:Tb-bdd-LpA}
\|\mathcal T_b\,r\|_{C^{1-\frac{1}{p}}(\bndry\domain)}\lesssim \|r\|_{L^p(\bndry\domain,\sigma)}.
\end{equation*}

\item[\tt(ii.)]\quad $\mathcal T_b$ is bounded: $L^p(\bndry\domain, \sigma) \to W^{1,p}(\bndry\domain,\sigma)$. Namely
 \begin{equation}\label{E:Tb-bdd-Lp}
 \|\mathcal T_b\,r\|_{L^p(\bndry\domain,\sigma)} + \|\partial_T(\mathcal T_b\,r)\|_{L^p(\bndry\domain,\sigma)}\lesssim \|r\|_{L^p(\bndry\domain,\sigma)}, 
 \end{equation}
 and $\mathcal T_br $ gives the unique solution to  \begin{equation}\label{ode3}
  -i\partial_{T}h(\zeta) + b(\zeta)h(\zeta) = r(\zeta)\ \  \text{for}\ \ \sigma-\text{a.e. } \zeta \in \bndry\domain.
\end{equation}
 \item[\tt(iii.)]\quad  $\mathcal T_b$ takes $\reu^{p}(\bndry\domain) $ into $h^{1,p}(\bndry\domain)$. Hence  $\mathbf C_D\circ \mathcal T_b$ takes $\reu^{p}(\bndry\domain)$ into
 $\mathcal H^{1,p}(\domain)$.
  \end{itemize}
\end{prop}

\vskip0.1in

\begin{proof} 
 We adopt the shorthand
  $$\Bzj(\zeta): = e^{i\int\limits_{\gamma(z, \zeta)} b(\xi) d\sigma(\xi)}\quad \text{for}\quad \zeta\in \bndry\domain;$$ 
  $$ \tilde b_{0}: = e^{i\int\limits_{\bndry\domain} b(\xi) d\sigma(\xi)}\, ,$$ and  $$  {\tilde r}(z) : =  \int\limits_{\bndry\domain} \Bzj(\zeta)r(\zeta)d\sigma(\zeta),\ \ \ z\in \bndry\domain.  $$ Then the compatibility condition \eqref{E:compint} is equivalent to
\begin{equation} \label{co}
    \tilde b_{0} - 1\ne 0.
\end{equation}

Proof of {\tt (i.)}: We only need to show 
\begin{equation}\label{db}
    \|{\tilde r}\|_{C^{1-\frac{1}{p}}(\bndry\domain)}\lesssim \|r\|_{L^p(\bndry\domain, \sigma) }.
\end{equation} 
First, since $ \|\Bzj\|_{L^\infty(\bndry\domain,\sigma)}\le e^{\|b\|_{L^1(\bndry\domain, \sigma)}}\lesssim 1, $
\begin{equation}\label{ll}
    \sup_{\zeta\in \bndry\domain}|{\tilde r}(\zeta)|\lesssim \|r\|_{L^1(\bndry\domain, \sigma)}\lesssim \|r\|_{L^p(\bndry\domain, \sigma)}.
\end{equation}
For fixed $\zeta_1 , \zeta_2 \in \bndry\domain$, the arclength  $\sigma(\gamma(\zeta_1, \zeta_2)) \approx |\zeta_1-\zeta_2|$ by  the Lipschitz property of $\bndry\domain$. Without loss of generality, assume that  $|\zeta_1-\zeta_2|$ is small. For $\sigma$-a.e. $\xi\in \bndry\domain\setminus \gamma(\zeta_1, \zeta_2)$,
$$ \left| e^{i\int\limits_{\gamma(\zeta_1, \xi)} b(\eta) d\sigma(\eta)}- e^{i\int\limits_{\gamma(\zeta_2, \xi)} b(\eta) d\sigma(\eta)}\right| \le e^{\|b\|_{L^1(\bndry\domain, \sigma)} } \left| \int\limits_{\gamma(\zeta_1, \zeta_2)}| b(\eta)| d\sigma(\eta)\right|\lesssim |\zeta_1 -\zeta_2|^{1-\frac{1}{p}}.$$
 Here we used the mean-value theorem  in the first inequality (more precisely, $|e^{z_1}-e^{z_2}|\le e^x|z_1-z_2|$ whenever $|z_1|, |z_2|<x: = \|b\|_{L^1(\bndry\domain, \sigma)}$) and   H\"older inequality  in the second inequality. On the other hand, for $\sigma$-a.e. $\xi\in \gamma(\zeta_1, \zeta_2) $, $$ \left|e^{i\int\limits_{\gamma(\zeta_k, \xi)} b(\eta) d\sigma(\eta)}\right|\le e^{\|b\|_{L^1(\bndry\domain, \sigma)}}\lesssim 1, k=1, 2.$$ Thus by H\"older inequality
\begin{equation*}
    \begin{split}
    |{\tilde r}(\zeta_1)-{\tilde r}(\zeta_2)|\lesssim   &  \int\limits_{\bndry\domain\setminus \gamma(\zeta_1, \zeta_2)}\left| e^{i\int\limits_{\gamma(\zeta_1, \xi)} b(\eta) d\sigma(\eta)}- e^{i\int\limits_{\gamma(\zeta_2, \xi)} b(\eta) d\sigma(\eta)}\right||r(\xi)|d\sigma(\xi) \\
    &+ \int\limits_{\gamma(\zeta_1, \zeta_2)}|r(\xi)|d\sigma(\xi)\\
    \lesssim &|\zeta_1 -\zeta_2|^{1-\frac{1}{p}}\|r\|_{L^p(\bndry\domain, \sigma) }.
    \end{split}
\end{equation*}
Equation \eqref{db} follows from the above inequality and \eqref{ll}.
\vskip0.2in
Proof of {\tt (ii.)}: We first show that  there exists a unique solution to \eqref{ode3}. 
 We shall adopt the arclength variable $s\in [0, s_{0}) $ to parametrize $\bndry\domain$, where $\zeta(0)=\lim_{s\to s_{0}^-} \zeta(s)=z$  for a fixed $z\in \bndry\domain$,   and $|\zeta'(s)|=1$ for a.e. $s\in [0, s_{0})$.
 
  Write $h(s): =h(\zeta(s))$,  and similarly for $T$, $b$, $\Bzj$ and $r$. In particular, for   $s\in [0,s_{0})$, we have $\Bzj(s)=e^{i\int\limits_{0}^s b(s) ds}$. It is immediate to  verify that $\Bzj, \Bzj^{-1}\in W^{1, p}((0, s_0))$. We further continuously extend $\Bzj$ to $s_{0}$ with $\Bzj(s_{0}):=e^{i\int\limits_{0}^{s_{0}} b(s) ds}= \tilde b_{0}$. 

Since $\dee_T h = \frac{dh}{ds}$, any solution $h$ to \eqref{ode3} necessarily satisfies
\begin{equation}\label{ode2}
    -ih'(s)+b(s)h(s)=r(s), \ \text{for a.e.}\ s\in (0, s_{0}).
\end{equation} 
This is a  first order linear ordinary differential equation. Using the method of  integrating factors, we have 
\begin{equation}\label{ode}
     \frac{d}{ds} \left(\Bzj(s) h(s)\right)= i \Bzj(s) r(s), \ \text{for a.e.}\ s\in (0, s_{0}).
\end{equation}
 Thus any solution to \eqref{ode2} is of the form
\begin{equation}\label{so}
\begin{split}
        h(s) =&i\tilde b^{-1}_{z}(s)\int\limits_0^s  \tilde b_{z}(t)r(t)dt+ C\tilde b^{-1}_{z}(s)\ \text{for a.e.}\ s\in (0, s_{0}) 
        \end{split}
\end{equation}
for some constant $C$. Since $ \tilde b_{z}r\in L^p((0, s_0))$, one has   $ \int\limits_0^s  \tilde b_{z}(t)r(t)dt\in  W^{1, p}((0, s_0))\subset L^\infty((0, s_0))$. Together with the fact that $\Bzj^{-1}\in W^{1, p}((0, s_0)) $, one can further  verify that $h   \in W^{1, p}((0, s_0))$. In particular, by the Sobolev embedding theorem, every solution   to \eqref{ode2}, and hence to \eqref{ode3},  is continuous on $(0, s_0)$. Note that due to  the closedness of $\bndry\domain$,  a continuous solution to \eqref{ode2} becomes  a solution to \eqref{ode3} if and only if the following  natural  boundary value condition  is imposed  to   \eqref{ode2}:
\begin{equation}\label{rb}
    h(0)=h(s_{0}).
\end{equation}
Making use of the facts that  $\Bzj(0) = 1$ and $\Bzj(s_{0}) = \tilde b_{0}$,  \eqref{rb} is further equivalent to
\begin{equation}\label{rb2}
    i+C = i \tilde b_{0}\int\limits_0^{s_0}  \tilde b_{z}(t)r(t)dt + C\tilde b_{0}.
\end{equation} 
Due to \eqref{co},  the constant $C$ in \eqref{so} is uniquely determined by \eqref{rb2}. As a consequence, there exists a unique solution  to   \eqref{ode2} with boundary value condition \eqref{rb}, and thus a unique solution to    \eqref{ode3}.

 Next, we verify that $h=\mathcal T_br$ in \eqref{E:auxOp} solves \eqref{ode3} explicitly. Integrate   \eqref{ode} from $0$ to $s_{0}$ to get
$$ \Bzj(s_{0})h(s_{0}) -  \Bzj(0)h(0) =\int\limits_{0}^{s_{0} }i \Bzj(s)r(s)ds=i\int\limits_{\bndry\domain} \Bzj(\zeta)r(\zeta)d\sigma(\zeta).  $$
Making use of the facts that $h(s_{0}) = h(0)=h(z)$, $\Bzj(0) = 1$ and $\Bzj(s_{0}) = \tilde b_{0}$ again, we see that
$ \mathcal T_br$  solves \eqref{ode3}.

Finally we show the desired estimate for $\mathcal T_b$. The estimate   $ \sup_{\zeta\in \bndry\domain}|{ \mathcal T_br}(\zeta)|$ is done in \eqref{sup}. 
For the estimate of $\partial_T(\mathcal T_br)$, we make use of  the fact that  $\mathcal T_b r$ satisfies \eqref{ode3} and \eqref{sup} to get
\begin{equation}\label{bdd}
\begin{split}
        \|\partial_T(\mathcal T_br)\|_{L^p(\bndry\domain, \sigma)}\le & \|b\mathcal T_br\|_{L^p(\bndry\domain, \sigma)} + \|r\|_{L^p(\bndry\domain, \sigma)}\\
        \lesssim& \|b\|_{L^p(\bndry\domain, \sigma)}\sup_{\zeta\in \bndry\domain}|\mathcal T_b r(\zeta)| + \|r\|_{L^p(\bndry\domain, \sigma)}\\
        \lesssim& \|b\|_{L^p(\bndry\domain, \sigma)}\|r\|_{L^p(\bndry\domain, \sigma)} + \|r\|_{L^p(\bndry\domain, \sigma)}\lesssim \|r\|_{L^p(\bndry\domain, \sigma)}.
        \end{split}
\end{equation}
Conclusion {\tt (ii.)} is thus proved.

 \vskip0.2in
Proof of {\tt (iii.)}: We only need to prove that $\mathcal T_b \reu^{p}(\bndry\domain)\subset h^{1,p}(\bndry\domain)$. Given $r\in \reu^{p}(\bndry\domain)$,  by definition of $\reu^{p}(\bndry\domain)$  there exists some $h\in h^{1,p}(\bndry\domain)$ such that $h$ satisfies   \eqref{ode3}. Due to the uniqueness of solutions to   \eqref{ode3} as proved in part {\tt (ii.)}, $\mathcal T_b r$ must be equal to $ h$ and thus $ \mathcal T_b r \in h^{1,p}(\bndry\domain)$. 
\end{proof}

\subsection{Proof of Theorem \ref{T:RDbar}} 
Let $p\geq 1$ and $r\in L^p(\bndry\domain, \sigma)$.
We first show that the holomorphic Robin problem \eqref{RdbarIntro}
 is solvable if and only if $r\in\reu^{p}(\bndry\domain)$. To see this, suppose that $G$ is a solution to \eqref{RdbarIntro} with datum $r$: then it is immediate from \eqref{RdbarIntro} and \eqref{D:H1q}
 that  $G\in \Hs^{1, p}(\domain)$, hence $\dot{G}\in h^{1,p}(\bndry\domain)$. Moreover \eqref{RdbarIntro}, \eqref{E:normal-deriv-hol} and \eqref{E:tan-der} give that
 \begin{equation}\label{ode1}
  -i\partial_{T}\dot{G} (\zeta) + b(\zeta)\dot{G}(\zeta) = r(\zeta)\ \  \text{for}\ \ \sigma-\text{a.e. } \zeta \in \bndry\domain\, ,
\end{equation}
proving that $r\in\reu^{p}(\bndry\domain)$, see \eqref{D:Rdata}. 
Conversely, if $r\in\reu^{p}(\bndry\domain)$ then it satisfies \eqref{ode1} for some $G\in \Hs^{1, p}(\domain)$, see \eqref{D:Rdata} and \eqref{D:h1q}, and it follows from \eqref{ode1},  \eqref{E:normal-deriv-hol}, \eqref{E:tan-der},  \eqref{D:H1q} and \eqref{D:Hp} that $G$ solves \eqref{RdbarIntro} with datum $r$. 
\vskip0.1in
 
Next, we observe that if $G\in \Hs^{1,p}(\domain)$ is a solution of \eqref{RdbarIntro} with datum $r\in \reu^p(\bndry\domain)$ then $G$ is also a solution of the regularity problem \eqref{E:D-Dbar-reg} with datum
$$f(\zeta):= \mathcal T_b\, r(\zeta)$$
(which is in $h^{1,p}(\bndry\domain)$ by Proposition \ref{P:Top} {\tt (iii.)}).   
 Hence the uniqueness and the representation formula: $G = (\mathbf C_D\circ \mathcal T_b) r$ follow from Theorem \ref{T:D-Dbar-reg}.  

We are left to prove the  estimate  \eqref{E:Rqbound} for $1<p<\infty$. Since   $\mathcal T_br\in C^{1-\frac{1}{p}}(\bndry\domain)$ by Proposition \ref{P:Top} (i.),    one has 
$G := (\mathbf C_D \circ \mathcal T_b) r \in C^{1-\frac{1}{p}}(\overline\domain)\cap \vartheta (D)$ 
with 
$$\|G\|_{ C^{1-\frac{1}{p}}(\overline\domain) }\lesssim \|{\mathcal T_b r}\|_{C^{1-\frac{1}{p}}(\bndry\domain)}. $$
In particular,  
  by the above and \eqref{db}  
$$\|G^*\|_{L^\infty(\bndry\domain, \sigma)  }\le  \sup_{z\in\domain}|G|  \lesssim \ \|{\mathcal T_b r}\|_{C^{1-\frac{1}{p}}(\bndry\domain )}\, \lesssim \, \|r\|_{L^p(\bndry\domain, \sigma) }. $$ 
Moreover, since $(G')^* \in L^p(\bndry\domain, \sigma)$ by assumption, we have $G'\in \Hs^p(\domain)$. Thus by  Lemma \ref{LaplacianLemma}, \eqref{E:tan-der} and \eqref{bdd}  \begin{equation*}
    \begin{split}
       \|(G')^*\|_{L^p(\bndry\domain, \sigma)}\lesssim& \| \dot{(G')}\|_{L^p(\bndry\domain, \sigma)}  =\|\dee_T {\mathcal T_b r}\|_{L^p(\bndry\domain, \sigma)} 
       \lesssim  \|r\|_{L^p(\bndry\domain, \sigma)}.
    \end{split}
\end{equation*}
\qed

\begin{rem}

As an immediate consequence of Theorem \ref{T:RDbar},  the data space for the holomorphic Robin problem can be characterized equivalently as  $$\reu^p_b(\bndry\domain) =\left\{ r\in L^p(\bndry\domain, \sigma):   {\mathcal T_b r}\in  h^{1,p}(\bndry\domain)\right\},$$
where ${\mathcal T_b r}$ is given by   \eqref{E:auxOp}. On the other hand, from the proof of Proposition \ref{P:Top}, given $r\in \reu^p_b(\bndry\domain)$, the unique solution $G$ also lies in $ C^{1-\frac{1}{p}}(\overline\domain)$.
\end{rem}

\fontsize{11}{11}\selectfont

\vspace{0.7cm}

\noindent williamgryc@muhlenberg.edu,

\vspace{0.2 cm}

\noindent Department of Mathematics and Computer Science, Muhlenberg College, Allentown, PA, 18104, USA.\\

\noindent loredana.lanzani@gmail.com, 

\vspace{0.2 cm}

\noindent Department of Mathematics, Syracuse University, Syracuse, NY, 13244, USA.

\vspace{0.2 cm}

\noindent Department of Mathematics, University of Bologna, Italy.

\noindent \\

\noindent jue.xiong@colorado.edu,

\vspace{0.2 cm}

\noindent Department of Mathematics, University of Colorado, Boulder, CO, 80309, USA.
\noindent \\

\noindent zhangyu@pfw.edu,

\vspace{0.2 cm}

\noindent Department of Mathematical Sciences, Purdue University Fort Wayne, Fort Wayne, IN 46805-1499, USA.\\

\end{document}